\documentclass[11pt]{amsart}
\usepackage{amsthm}
\usepackage{amsmath}
\usepackage{amsfonts}
\usepackage{amssymb}
\usepackage{color}

\newcommand{\beq}{\begin{equation}}
\newcommand{\beqn}{\begin{equation*}}

\newcommand{\eeq}{\end{equation}}
\newcommand{\eeqn}{\end{equation*}}

\newcommand{\n}{\mathbb{N}}
\newcommand{\R}{\mathbb{R}}

\newcommand{\rk}{{\rm rk}}

\newcommand{\sone}{S^1}
\newcommand{\z}{\mathbb{Z}}
\newcommand{\Z}{\mathbb{Z}}
\newcommand{\Q}{\mathbb{Q}}

\newtheorem{theorem}{Theorem}
\newtheorem{lemma}{Lemma}

\newtheorem{definition}{Definition}
\newtheorem{proposition}{Proposition}
\newtheorem{remark}{Remark}
\newtheorem{example}{Example}
\title{Closed geodesics on connected sums and $3$-manifolds}
\author{Hans-Bert~Rademacher}
\address{Mathematisches Institut, Universit{\"a}t Leipzig,
04081 Leipzig, Germany}
\email{rademacher@math.uni-leipzig.de}
\author{Iskander~A.~Taimanov}
\address{Sobolev Institute of Mathematics,
630090 Novosibirsk, and Novosibirsk State University,
630090 Novosibirsk, Russia}
\email{taimanov@math.nsc.ru}
\begin{document}
\begin{abstract}
We study the asymptotics of the number $N(t)$
of geometrically
distinct closed geodesics
of a Riemannian or Finsler metric
on a connected sum of
two compact manifolds of dimension at least three
with non-trivial fundamental groups
and apply this result to the prime decomposition of
a three-manifold.
In particular we show that the function $N(t)$
grows at least like the prime numbers on
a compact $3$-manifold with infinite fundamental group.
It follows that a generic Riemannian metric on a
compact $3$-manifold has infinitely many
geometrically distinct closed geodesics.
We also consider the case of a connected sum of a
compact manifold with positive first Betti number and
a simply-connected manifold which is not homeomorphic to
a sphere.
\end{abstract}
\maketitle
\section{Introduction}
For a compact manifold
with infinite fundamental group
endowed with a Riemannian
or Finsler metric we investigate the
asymptotics of the number $N(t)$
of geometrically distinct closed geodesics
with length $\le t.$
In any non-trival free
homotopy class there exists a non-trivial closed geodesic.
Therefore not the growth of the fundamental group but
the growth of the number of conjugacy classes allows
to estimate of the function $N(t).$
If the first Betti number $b_1(M)=\rk_{\z} H_1(M;\z)$
satisfies
$k=b_1(M)\ge 2$ then one obtains
$\liminf_{t \to \infty} N(t)/t^k>0,$
in particular there are infinitely many closed geodesics.
Estimates for the function $N(t)$ in case of an
infinite cyclic fundamental group are due
to Bangert and Hingston~\cite[Thm.]{BH},
for an infinite solvable fundamental group
due to~\cite[Thm.3]{T1993}, for products
of $S^1$ with a simply-connected manifold due
to Gromov~\cite[p.398]{Gr} and
for an almost nilpotent,
but not infinitely cyclic fundamental
group due to Ballmann~\cite[Satz 2]{Ball}.
Related results were also obtained by Tanaka~\cite{T}.
Ballmann, Thorbergsson and Ziller~\cite{BTZ1981}
present results for fundamental groups in which
there is a non-trivial element for
which two different powers
are conjugate to each other.

In Section~\ref{sec:normal}
we study the growth function
$F(k)$ of the number of conjugacy
classes in $G$  which can be
represented by words of length $\le k$
in some generating set
for a free product $G=G_1\ast G_2$
of non-trivial groups $G_1$ and $G_2.$
We apply this result in Section~\ref{sec:conn}
to the fundamental group
$\pi_1\left(M_1\#M_2\right)\cong\pi_1(M_1)\ast \pi_1(M_2)$
of the connected sum
$M_1 \# M_2$ of
compact manifolds $M_1$ and $M_2$ of dimension
$\ge 3.$ Section~\ref{sec:conn} and
Section~\ref{sec:connA} give the proof of
\begin{theorem}
\label{thm:conn-intro}
Let $M_1$ and $M_2$ be two compact manifolds of dimension
$n \ge 3$ and let $\pi_1(M_1)\not=0$
and let $M_2$ be not homeomorphic to the $n$-sphere.
Then the number of
geometrically distinct
closed geodesics $N(t)$ of length
$\le t$ of a Riemannian or Finsler metric on the
connected sum $M=M_1 \# M_2$ satisfies the following:
\begin{itemize}
\item[(a)] If $\pi_1(M_2)\not=0$ or
if the first Betti number $b_1(M_1)=\rk_{\z}H_1(M_1;\z)\ge 1$
then
$\liminf_{t\to \infty}N(t)\log(t)/t>0.$
\item[(b)] If $\pi_1(M_2)\not=0$
and $\pi_1(M_2)\not\cong\z_2$ then
$\liminf_{t\to \infty} \log (N(t))/t>0.$
\end{itemize}
In particular in both cases there are infinitely
many geometrically distinct closed geodesics.
Moreover these inequalities hold if only non-contractible closed geodesics are counted.
\end{theorem}
We remark (see Section~\ref{sec:remark}) that together with known results Theorem~\ref{thm:conn-intro} implies that the only remaining case for which the existence of infinitely many geometrically distinct closed geodesics on connected sums is not established is as follows: {\sl one of of manifolds, say $M_2$, is simply-connected and $b_1(M_1) =0$ with $\pi_1(M_1)$ infinite}.

Note that Paternain and  Petean proved in \cite{PP} that
for a bumpy Riemannian metric on a
connected sum $M=M_1 \# M_2$
the number $N^0(t)$ of {\em contractible}
closed geodesics of length $\le t$
has exponential growth under the following
assumption: $\pi_1(M_1)$ has a subgroup of
finite index $\ge 3$ and $M_2$ is simply-connected and
not a homotopy sphere.

Theorem~\ref{thm:conn-intro}
is applied in Section~\ref{sec:three} to the
prime decomposition of a three-dimensional manifold:
\begin{theorem}
\label{thm:three-dimension}
Let $M$ be a compact $3$-manifold.
If the fundamental group $\pi_1(M)$ is infinite
then for any Riemannian or Finsler metric the
number $N(t)$ of geometrically distinct closed geodesics
of length $<t$ satisfies:
$\liminf_{t\to \infty} N(t)\log(t)/t>0.$
In particular there are infinitely many
geometrically distinct closed geodesics.
\end{theorem}
We also give assumptions in Theorem~\ref{thm:three-dim-cases}
under which the growth
of $N(t)$ on a three-dimensional compact manifold is
exponential or polynomial of arbitrary order.
The function $N(t)$ has exponential growth if the
$3$-manifold is neither prime nor diffeomorphic to
the connected sum
$\R P^3\#\R P^3$ of two real-projective spaces $\R P^3.$
Using a result by Hingston~\cite[6.2]{Hi} applied to
a Riemannian metric on the $3$-sphere
all of whose closed geodesics are hyperbolic
Theorem~\ref{thm:three-dimension} implies
\begin{theorem}
\label{thm:three-dimensionA}
Let $M$ be a compact $3$-manifold.
Then for a
$C^4$-generic Riemannian
metric
on $M$
the
number $N(t)$ of geometrically distinct closed geodesics
of length $<t$ satisfies:
$\liminf_{t\to \infty} N(t)\log(t)/t>0.$
In particular there are infinitely many
geometrically distinct closed geodesics.
\end{theorem}
The proof is presented in Section~\ref{sec:three}.
It is shown in \cite{Ra1} that a
$C^4$-generic Riemannian metric on a compact and
manifold with trivial or finite
fundamental group carries
infinitely many geometrically distinct closed
geodesics.
Therefore in this paper we concentrate on
compact manifolds with {\em infinite fundamental group.}

 For surfaces the corresponding statements holds
for an arbitrary Riemannian metric, cf. Remark~\ref{rem:surface}.
 \section*{Acknowledgements}
 We are grateful to Bernhard Leeb and Misha Kapovich for
 helpful discussions about the geometry and topology of
 three-manifolds.
\section{Normal forms in free products of groups}
\label{sec:normal}
We use the following terminology:
As a general reference for this section
we use Lyndon and Schupp's book on combinatorial group
theory, cf.~\cite[Sec. V.9]{LS}.
We consider the free product $G=G_1* G_2$
of two non-trivial groups $G_1,G_2.$
Any element
$g \in G=G_1 \ast G_2, g\not=1$ can be expressed
uniquely in {\em normal form} as a product of non-trivial elements
of the groups $G_1, G_2:$
\begin{equation*}
g=g_1 g_2 \cdots g_r\,, r\ge 1\,.
\end{equation*}
Here no two consecutive factors $g_j,g_{j+1}$
belong to the same group $G_j.$
Then $r=|g|$ is the {\em length} of $g.$
If there are two elements $g=g_1\cdots g_r c_1\cdots c_t$
and
$h=c_t^{-1}\cdots c_1^{-1} h_1\cdots h_s$ in normal form
with $h_1\not=g_r^{-1}$ then the letters
$c_1,\ldots,c_t$ {\em cancel} in the the product $gh.$
If $g_r$ and $h_1$ do not lie in the same factor $G_j$
the normal form of the
product is given by:
$g_1\cdots g_r h_1\cdots h_s.$
If $g_r$ and $h_1$ lie in the same factor $G_j$ and
$a=g_r h_1\not=1$ then the normal form of the product $gh$
is given by
$g_1\cdots g_{r-1} ah_2\cdots h_t.$
Then we say that $g_r$ and $h_1$ have been
{\em consolidated}. I.e. their product $a$ gives
a single element in the normal form of $gh.$

We call $g$ with normal form $g=g_1\cdots g_r$
{\em cyclically reduced} if either
$|g|\le 1$ or if $g_1$ and $g_r$ do not lie in the same factor
$G_j.$
Since we only consider a product of two factors this implies that
$r$ is even.
The element $g$ is called
{\em weakly reduced} if $|g|\le 1$ or $g_1\not=g^{-1}_r.$
A {\em cyclic permutation} of the normal form
$g=g_1\cdots g_r$ is of the form
$g'=g_s \cdots g_r g_1\cdots g_{s-1}$
for some $s\le r.$

Later on we will use the following
\begin{lemma}
\label{lem:cyclically}
Let $g=g_1\cdots g_r$ be a cyclically reduced word in
the free product $G=G_1 * G_2, $ i.e. $r$ is even.
Then any element $g'=h g h^{-1}$ for some $h \in G$
which is also cyclically reduced is a cyclic permutation
of the word $g.$
Hence
$ g'=g_{s}\cdots g_r g_1 g_2\cdots g_{s-1}$
for some $s \le r.$
\end{lemma}
\begin{proof}
If $h=g_{s-1}^{-1}\cdots g_1^{-1}$ then
$h g h^{-1}=g_{s}\cdots g_r g_1\cdots g_{s-1}.$
For $h=h' g^{k}, k\in \mathbb{Z}$ we obtain
$h g h^{-1}= h' g h'^{-1}.$
Therefore we consider now the largest number $s \le r$
such that
$h$ has the following normal form:
$h=h_1\cdots h_t g_{s-1}^{-1}\cdots g_1^{-1}$
where $h_t$ lies in the same factor as $g_{s}$ and
$h_t g_{s}\not=1.$
Then the normal form of $g'=hgh^{-1}$ is
given by:
\begin{equation*}
h_1\cdots h_{t-1} a g_{s+1}\cdots g_r g_1
\cdots g_{s-1} h_t^{-1}\cdots h_1^{-1}\,, a=h_tg_{s}\,.
\end{equation*}
Hence this normal form has odd length $|g'|=r+2t-1,$ i.e.
it is not cyclically reduced. This also follows
from the fact that the first element $h_1$ and the last
element $h_1^{-1}$ of the normal form belong to the same factor.
\end{proof}
\begin{lemma}
\label{lem:permutation}
For any $g' \in G=G_1*G_2, g'\not=1$
there is an element $g \in G$
with cyclically reduced normal form
$g=g_1 \cdots g_r$ such that $g'$ and $g$ are
conjugate, i.e. there exists an element
$h \in G$ with $g'=h g h^ {-1}.$
And any element $g_*$ which is conjugate to $g'$
and has cyclically reduced normal form is
a permutation of $g.$
\end{lemma}
\begin{proof}
The normal form $g'=g_1\cdots g_r$ is uniquely determined.
If $r$ is even, this is already a cyclically reduced normal
form. Therefore we assume $r$ is odd, i.e. $g_1,g_r$ belong to the
same factor $G_j.$ If $g_r g_1=1$ we obtain an element
$g'_1=g_1^{-1}g'g_1=g_2\cdots g_{r-1}$
in normal form which is conjugate
to $g'.$ After finitely many steps we obtain an element
$g'_s=g_{s+1}\cdots g_{r-s}$ in normal form conjugate to
$g'$ with $g_{s+1},g_{r-s}$ lying in the same factor $G_j$
and satisfying $g_{r-s}g_{s+1}\not=1.$ Hence this is a
weakly reduced normal form. Then
$g:=g_{s+1}^{-1} g'_s g_{s+1}$ is conjugate to $g'$
and has normal form $g=g_{s+2}\cdots g_{r-s-1} a$
with $a=g_{r-s}g_{s+1}^{-1}.$ Hence this normal
form is cyclically reduced.
Now Lemma~\ref{lem:cyclically} implies that
this $g$ is uniquely determined up to permutation.
\end{proof}
For a finitely generated group $G$ with a finite
set $E$ of generators the {\em word length}
$w_E(g)$ of an element $g \in G$ is given as the
minimal number of elements $e_1,\ldots,e_k\in E$ such that
$g=e_1\cdots e_k.$
Let $G_1,G_2$ be finitely generated
groups, i.e. there exist
finite generating sets $E_j \subset G_j,j=1,2$
for $G_1$ and $G_2.$
We denote by $E$ the disjoint
union of $E_1$ and $E_2.$
This is a generating set for $G=G_1\ast G_2.$
For $g \in G=G_1*G_2$
with normal form $g=g_1 \cdots g_r$ the word length
can be expressed as follows:
$w_E(g)=w_{E}(g_1)+\ldots+w_{E}(g_r).$
\begin{proposition}
\label{pro:w}
Let $G=G_1*G_2$ be the free product of two non-trivial
and finitely generated groups $G_1,G_2$ with
generating sets $E_1,E_2$ with disjoint union $E.$
Let $\overline{G}=\{[g]; g\in G\}$ be the set of conjugacy classes,
i.e. $[g]=\{h g h^ {-1}\,;\, h \in G\}.$
We define the function
 $F(k):=\# \{[g] \in \overline{G}\,;\, g \in G; w_E(g)\le k\}.$
If at least one of the groups $G_1,G_2$ has
at least three elements,
then there is a constant $\lambda >1$ such that
$ F(k) \ge \lambda^k.$
\end{proposition}
\begin{proof}
Let $a$ be an element of the generating set $E_1$ for $G_1$ and
$b_1$ be an element from the generating set $E_2$ for $G_2.$
If $b_1^2\not=1,$ let $b_2:=b_1^2.$
In this case $w_E(b_2)=w_E(b_1^2)=2.$
Otherwise let $b_2$ be an element
from the set $E_2-\{b_1\}.$
This set is non-empty since by assumption
$G_2$ has at least three elements.

Let $r \ge 1$ and let $(m_1,m_2,\ldots,m_r)$
be an $r$-tuple with $1 \le m_j \le 2, j=1,\ldots,r.$
Then we define an element
\begin{equation}
\label{eq:gm}
g(m_1,\ldots,m_r)=a b_{m_1} a b_{m_2}\cdots a b_{m_r}.
\end{equation}
This is a cyclically reduced normal form since
$|g(m_1,\ldots,m_r)|=2r$ is even.
Hence $w_E\left(g(m_1,\ldots,m_r)\right)\le 3 r.$
We conclude from Lemma~\ref{lem:permutation} that the
conjugacy classes of $g(m_1,\ldots,m_r)$
and $g(m'_1,\ldots,m'_r)$ coincide if and only if
the tuples $(m_1,\ldots,m_r)$
and $(m'_1,\ldots,m'_r)$ coincide up to a cyclic permutation.
Hence
$$F(3r) \ge 2^r / r. $$
\end{proof}
\begin{remark}
\rm
(a)
The Proof also shows that $G_1\ast G_2$ has a free non-cyclic
subgroup of rank at least $2$ if $G_2$ has at least three elements.
The elements $ab_1$ and $a b_2$ are generators for the free subgroup
of rank $2.$

\medskip

(b) The free group $G=\langle a_1,\ldots,a_k\rangle$
of rank $k$ generated by $a_1,\ldots,a_k$
is isomorphic to the free product $\langle a_1 \rangle\ast
\ldots\ast \langle a_k\rangle.$
Hence the assumptions of Proposition~\ref{pro:w} are
satisfied for a non-cyclic free group $G.$
\end{remark}
\section{Closed geodesics on connected sums
$M_1\#M_2$ with $\pi_1(M_1)\not=0, \pi_1(M_2)\not= 0.$}
\label{sec:conn}
We apply the results of the last section to the fundamental
group of a connected sum
of two compact manifolds with non-trivial fundamental group
to estimate the number $N(t)$ of
geometrically distinct closed geodesics of length $\le t.$
Two closed geodesics $c_1,c_2:S^1\longrightarrow M$
are called {\em geometrically equivalent} if
the images coincide as point sets, i.e. $c_1(S^1)=c_2(s^1)$
and if in addition in the case of a non-reversible Finsler
metric also the orientations of the curves coincide.
Therefore in the estimates for the function $N(t)$
depending on the reversibility of the Finsler metric
a factor $2$ could occur in the estimates for the
function $N(t)$ which would not affect our estimates for
the asymptotics of $N(t).$

In every non-trivial
free homotopy class of closed curves
on a compact manifold
there is a
shortest closed curve which is a closed geodesic.
Two closed curves $c_1,c_2: [0,1]\longrightarrow M$
with $c_j(0)=c_j(1)=p, j=1,2$
are {\em freely homotopic} if their homotopy classes in
the fundamental group $\pi_1(M)=\pi_1(M,p)$ are conjugate.
We use that the fundamental group $\pi_1(M)$ of a connected
sum $M=M_1 \# M_2$ of two manifolds of dimension $n \ge 3$
is isomorphic to the free product of the fundamental groups
$\pi_1(M_j),j=1,2:$
\begin{equation}
\label{eq:fund-group-conn-sum}
\pi_1\left(M_1\# M_2\right)\cong
\pi_1\left(M_1\right)\ast \pi_1\left(M_2\right)\,.
\end{equation}
This follows from Seifert-Van Kampen's
theorem~\cite[ch.III Prop.9.4]{Br}.
Before we give the Proof of Theorem~\ref{thm:conn-intro}
we discuss examples:
\begin{example}
\rm
(a)
The {\em infinite dihedral group}
$\z_2 \ast \z_2$ is isomorphic to the fundamental group
$\pi_1(M)$ of
the connected sum $M=M_1 \# M_2$ of two manifolds
with $\pi_1(M)\cong \pi_1(M_2)\cong \z_2.$ Then
$M$ has a twofold covering $\widetilde{M}$ with
$\pi_1(\widetilde{M})\cong\z,$ since
$\z$ is a normal subgroup of $\z_2\ast \z_2$ of index $2.$
Hence the free product $\z_2 \ast \z_2$
is isomorphic to a semidirect product $
\z \ast \z \cong \z \rtimes \z_2,$
compare the following Proof.
For example the connected sum $M=\R P^n \# \R P^n$
of two real projective spaces
satisfies this assumption.

\smallskip

(b) The {\em modular group} ${\rm PSL}_2(\z)={\rm SL}_2(\z)/\z_2$
is isomorphic
to the free product $\z_2\ast \z_3.$
It is isomorphic to
the fundamental group of the connected sum
$\R P^{2m+1} \# S^{2m+1}/\z_3, m \ge 1$ of the real projective
space $\R P^ {2m+1}$ and the lens space
$S^{2m+1}/\z_3.$ It satisfies the assumption of Part (b).

\smallskip

(c) For $n \ge 3$
the free group of rank $k \ge 2$ is isomorphic to
the fundamental group of the connected sum
$M=\left(S^1 \times S^{n-1}\right)\#
\cdots \#
\left(S^1 \times S^{n-1}\right)$ of $k$ copies of
$S^1\times S^{n-1}.$
Then the assumption of Part (b) is satisfied.
\end{example}
Now we give the {\bf Proof of Theorem}~\ref{thm:conn-intro}
stated in the Introduction under the assumption
$\pi_1(M_1)\not=0, \pi_1(M_2)\not=0:$
\begin{proof}
(a) Assume $\pi_1=\pi_1(M)\cong \pi_1(M_2)\cong \z_2.$
Let $a$ and $b$ be the generators of the two copies of $\Z_2$,
hence
$\pi_1 \cong \langle a \rangle \ast \langle b \rangle$.
We put $t=ab$. Then $\pi_1$ is generated by $a$ and $t$ and
the following relation holds:
$ata^{-1} = t^{-1}$.
Therefore $\pi_1$ contains a subgroup which is generated by $t$,
is isomorphic to $\Z$ and whose index is equal to two,
in particular, it is a normal subgroup.

Then there is a twofold covering space
$\widetilde{M}$ with
$\pi_1(\widetilde{M}) = \langle t \rangle \cong \Z\,.$
The claim follows from
a result by Bangert and Hingston~\cite[Thm.]{BH}.

\smallskip

(b)
Let $a$ be an element of a generating set $E_1$ for $\pi_1(M)$ and
$b_1$ be an element from a generating set $E_2$ for $\pi_1(M_2).$
If $b_1^2\not=1,$ let $b_2=b_1^2.$ Otherwise let
$b_2$ be an element from $E_2-\{b_1\},$
which exists by assumption.
And let $r \ge 1$ and let $(m_1,m_2,\ldots,m_r)$ be an $r$-tuple with
$1 \le m_j \le 2.$

Let $\mathcal{M}$ be a maximal subset of the tuples
$(m_1,m_2,\ldots,m_r)$ with $1\le m_j \le 2$ and such that no
distinct tuples in $\mathcal{M}$ are a cyclic permutation of
one another. Then $\#\mathcal{M} \ge 2^r/r.$
The elements $g(m_1,\ldots,m_r)$
for $(m_1,\ldots,m_r)\in \mathcal{M}$ defined by
Equation~\eqref{eq:gm}
define at least $2^r/r$ pairwise distinct free conjugacy
classes, cf. the Proof of Proposition~\ref{pro:w}
and Lemma~\ref{lem:permutation}.

Let $L$ be the maximum
length of the shortest closed curves in
the homotopy classes $a, b_1,b_2 $ $\in \pi_1(M,p).$
Hence it is the maximum
of the lengths of shortest geodesic loops with base point
$p$ in the homotopy classes $a,b_1,b_2.$ Then
there exist closed geodesics
$\gamma (m_1,\ldots,m_r)$ for $(m_1,\ldots,m_r) \in
\mathcal{M}$
freely homotopic to $g(m_1,\ldots,m_r)$
which are pairwise distinct
and pairwise not freely homotopic.
We obtain
\begin{equation*}
 L(\gamma (m_1,\ldots,m_r)) \le 3r L\,.
\end{equation*}
Two distinct closed geodesics
$\gamma(m_1,\ldots,m_r)$ can be geometrically equivalent only
if they are coverings of the same prime closed geodesic.

Let $L_1$ be the length of a shortest
homotopically non-trivial closed geodesic on $M.$
Then the multiplicity of $\gamma (m_1,\ldots,m_r)$ is bounded
from above by $3r L /L_1.$ Hence we obtain for the number
$N(3rL)$
of geometrically distinct closed geodesics
of length $\le 3rL$ as lower bound:
\begin{equation*}
 N\left(3rL\right)\ge \frac{2^r L_1}{3r^2 L}\,.
\end{equation*}
This finishes the proof.
\end{proof}
\begin{remark}
\rm
The above Proof of Theorem~\ref{thm:conn-intro}
also shows that the following statements hold
under the assumption $\pi_1(M_1)\not=0,\not\cong \z_2$
and $\pi_1(M_2)\not= 0:$

\smallskip

(d) The fundamental group $\pi_1(M)$ has {\em exponential
growth.} By definition this means that for a generating set
$E$ the function
$\#\left\{g\in G; w_E(g)\le k\right\}$
has exponential growth.

\smallskip

(e) For points $p,q \in M$ denote by $n_t(p,q)$ the number of
geodesics of length $\le t$ joining $p$ and $q.$
Then
one can show with the results of Section~\ref{sec:normal}
that for all $p,q \in M:$
$\liminf_{t\to \infty} \log (n_t(p,q))/t>0\,.$
Then the {\em topological entropy} $h_{top}(\phi)$ of the
geodesic flow $\phi: T^1M \longrightarrow T^1M$ on the
unit tangent bundle
$T^1M$ is positive,
cf. \cite[Prop.4.1]{Ma}.
\end{remark}
\section{Closed geodesics on connected sums
$M_1\#M_2$ with $b_1(M_1)= 1$}
\label{sec:connA}

In this section we give the
{\bf Proof of Theorem}~\ref{thm:conn-intro}
stated in the Introduction under the assumption
$b_1(M_1)=1$ and $M_2$ not homeomorphic to
$S^n:$


\begin{proof}
It is enough to consider the case of $b_1(M_1)=\rk_{\z}H_1(M_1,\z) = 1,$
cf. Propostion~\ref{pro:virtual-number}.
Let $h \in \pi_1(M_1)$ and the corresponding homology class $[h]$ be of infinite order, i.e.,
generate $H_1(M_1)/\mathrm{Torsion}$:
$$
H_1(M_1) = \Z[h] \oplus \mathrm{Torsion}.
$$
\begin{proposition}
\label{pro:pk}
For $k \leq n-2$
$$
\pi_k(M_1 \# M_2) = \pi_k(M_1 \vee M_2).
$$
\end{proposition}

{\sc Proof of Proposition}. The manifold $M$ is obtained by removing $n$-discs $D_1$ and $D_2$ from $M_1$ and $M_2$ and gluing to the boundaries the cylinder $C = S^{n-1} \times [0,1]$. We take cell decompositions of $\partial D_1 = $ and $\partial D_2$ into unions of $(n-1)$-cells and the points $p_1$ and $p_2$.
These cell decompositions are extended to the cell decomposition of $C$ by adding to them $1$-cell joining $p_1$ and $p_2$ and an $n$-cell $D_0$. Let us now extend the cell decomposition of $C$ to a cell decomposition of $M$.

Any map of $l$-dimensional CW complex $X$ into $M$ is homotopic to a map into the $l$-skeleton $M^{(l)}$ of $M$. Therefore any map of a CW complex of dimension $\leq k+1 < n$  into $M$ is homotopic to a map into $M \setminus D_0$ and this imply that
$$
\pi_k(M_1 \# M_2) = \pi_k(M \setminus D_0).
$$
Moreover the same is valid if we add $n$-discs to $\partial D_1$ and $\partial D_2$ in $M \setminus D_0$. The resulted complex $Y$ is homotopic to the manifolds $M_1$ and $M_2$ which are connected by an interval attached to the interval attached to $p_1 \in M_1$ and $p_2 \in M_2$. By shrinking this interval to a point we obtain the bouquet $M_1 \vee M_2$ which is homotopically equivalent to $Y$. This proves
Proposition~\ref{pro:pk}
\begin{proposition}
\label{pro:wedge}
Let $M_2$ be $(k-1)$-connected with $k \geq 2$. Then
$$
\pi_1(M_1 \vee M_2) = \pi_1(M) = \pi_1(M_1),
$$
$$
\pi_k(M_1 \vee M_2) = A_1 \oplus A_2, \ \ A_1 = \pi_k(M_1), \ \ A_2 = \Z[\pi_1] \otimes_\Z H_k(M_2),
$$
$A_1$ and $A_2$ are submodules over $\Z[\pi_1]$,
the action of $\pi_1$ on $A_1$ is induced by the action of $\pi_1(M_1)$ on $\pi_k(M_1)$
and $\pi_1$ acts on $A_2$ by multiplications in $\Z[\pi_1]$.
\end{proposition}
{\sc Proof of Proposition.} The universal cover $S^k \to \widetilde{M_1 \vee M_2}$ of $M_1 \vee M_2$ is the bouquet of the universal cover
$\widetilde{M_1}$ of $M_1$ with infinitely many copies of $M_2$ parameterized by elements of $\pi_1$.
In homotopical calculations we may assume that $M_2$ is a CW complex with one point, with no other cells of dimension $<k$ and with finitely many cells of dimension $k$ which are one-to-one correspondence with generators of $H_k(M_2)=\pi_k(M_2)$.

Let us take a map $f: S^k \to \widetilde{M_1 \vee M_2}$. By homotopy we reduce this map to a form when $S^k$ is decomposed into finitely many domains such that each domain is mapped into one of the copies of $M_2$ or into $\widetilde{M_1}$, and the boundaries of these domains are mapped into the point $p=M_1 \cap M_2$. If some domain $W$ is mapped into $gM_2$ then by adding to $[S^k]$
a certain element of $g\pi_k(M_2)$ we change the map $f$ to another one which  is homotopic to the map which is the same outside $W$ and $W$ is mapped into $p=M_1 \cap M_2$. Algebraically that means that $[f]$ is replaced by $[f]-g a_g$ where $a_g \in H_k(M_2)=\pi_k(M_2)$. This element $g a_g$ is as follows. Let the $k$-skeleton of $M_2$ is the bouquet of $k$-spheres $S_1^k,\dots,S_m^k$. We denote by $gS^k_1,\dots,gS^k_m$ the corresponding  spheres in the copy of $M_2$ marked by $g \in \pi_1$. Let $d_{g,i} =
\deg [S^k \stackrel{f}{\rightarrow} gS^k_i], i=1,\dots,m$. Then
$$
g a_g = \sum_{i=1}^m d_{g,i} [gS^k_i].
$$
By successively applying this procedure we represent $[f]$ as the sum
$$
[f] = s_1 + s_2, \ \ s_1 \in \pi_k(M_1), \ \ s_2 \in \Z[\pi_1] \otimes_\Z H_k(M_2).
$$
This decomposition is clearly unique because $s_2$ realizes via the Hurewicz homomorphism nontrivial elements in a component of $H_k$ on which $\pi_1$ acts by multiplications, and the Hurewicz homomorphism maps $s_1$ into another component of $H_k$.
The actions of $\pi_1$ on $A_1$ and $A_2$ are clear.
Proposition~\ref{pro:wedge} is established.

Let $h \in \pi_1(M)$ be realized by a loop $\omega$ based at $p \in M$ and $\Lambda[h](M,p)$ be the space of $S^1$-maps into $M$ freely homotopic to $h$. By
\cite[Abschn.2]{Ba} or
\cite[Lem.1]{T1985}, the exact homotopy sequence for the Serre fibration
$$
\Lambda[h](M) \stackrel{\Omega[h](M)}{\longrightarrow} M
$$
which maps the loop into the marked point includes the following piece:
$$
\pi_k(M) \stackrel{h_\ast-1}{\longrightarrow} \pi_{k-1}(\Omega(M)) = \pi_k (M) \to \pi_{k-1}\Lambda[h](M)
$$
where $h_\ast: \pi_k \to \pi_k$ is the action of $h \in \pi_1$ on the $k$-th homotopy group $\pi_k$.
Hence if $(h_\ast-1)$ is not an epimorphism then the group $\pi_{k-1} \Lambda[h](M)$ is nontrivial.
In special case $\pi_1=\Z$ that was found by Bangert--Hingston who proved that

{\it if in this situation the powers of $h$ induce different free homotopy classes then the number $N(t)$ of simple closed geodesics of length $\leq t$ grows as
$$
C \frac{t}{\log t} \ \ \ \mbox{for some constant $C>0$}
$$
and moreover such an inequality holds if only non-contractible simple closed geodesics are counted.}

Let us prove that under assumptions of
Theorem~\ref{thm:conn-intro}
$h_\ast-1: \pi_k(M)\to \pi_k(M)$ is not an epimorphism.
We do that by contradiction. Let $(h_\ast-1)$ is an epimorphism which implies, by Proposition~\ref{pro:pk}, that
the multiplication
$$
\times(h-1): \Z[\pi_1]\otimes_\Z H_k(M_2) \to \Z[\pi_1]\otimes_\Z H_k(M_2)
$$
is an epimorphism. Let us take a generator $\gamma$ of $H_k(M_2)=\pi_k(M_2)$.
Denote by the same symbol, a generator of $\pi_k(M)$ which is realized by
the corresponding map $S^k \to \widetilde{M}$ to a fixed $M_2$-component of the bouquet.
If $\gamma \in (h_\ast-1)\pi_k(M)$, then there exists and element
$$
q\gamma, \ \ \ q \in \Z[\pi_1],
$$
such that $(h-1)q \gamma = \gamma$, and hence
$$
(h-1)q = 1
$$
in $\Z[\pi_1]$ if $\gamma$ is of infinite order, or in $\Z_N[\pi_1]$ if $\gamma$ is of order $N < \infty$.
Take the homomorphism of $\pi_1$ into $\pi_1/[\pi_1,\pi_1]$ which induces the homomorphism of the group rings
and derive that
$$
([h]-1)[q] = 1 \ \ \mbox{in $\Z[H_1]$, or in $\Z_N[H_1]$}.
$$
If $H_1(M)$ has a torsion we may again transfer to another ring $\Z[H_1/\mathrm{Torsion}] = \Z[u,u^{-1}]$, or to
$\Z_N[H_1/\mathrm{Torsion}] = \Z_N[u,u^{-1}]$ and derive the equality
$$
(u-1)\bar{q} = 1
$$
which has to hold in this ring.
Let
$$
\bar{q} = a_r u^r + \dots + a_s u^s, -\infty < r \leq s < \infty.
$$
Then
$$
(u-1)\bar{q} = -a_ru^r + \dots a_s u^{s+1} \neq 1,
$$
and we arrive at contradiction. Hence $\gamma \notin (h_\ast-1)\pi_k(M)$.
Now we apply arguments from \cite{BH} to derive the conclusion.
\end{proof}

\section{A remark on closed geodesics on connected sums of manifolds}
\label{sec:remark}

For a connected sum of two closed manifolds $M_1 \# M_2$
of dimension $n \ge 3$ which are not homeomorphic to the
$n$-sphere
it is known that
every metric on it has infinitely many
geometrically distinct closed geodesics if

\begin{itemize}

\item[(a)]
one of them, say $M_2$ is simply-connected, and the fundamental group of
another one, $\pi_1(M_1)$, is finite;

\item[(b)]
one of them, say $M_2$, is simply-connected
and $M_1$ has
has positive first Betti number,
see Theorem~\ref{thm:conn-intro} Part(a);

\item[(c)]
both $M_1$ and $M_2$ are not simply-connected,
see Theorem~\ref{thm:conn-intro} Part(b).
\end{itemize}

The remaining case for which there is no
positive answer is as follows:
{\sl
$M_2$ is simply-connected and $\pi_1(M_1)$ is infinite with
$b_1(M_1)=0.$}

We have to justify only on the case (a), because two other cases follow from Theorem~\ref{thm:conn-intro}.

By the Gromoll--Meyer theorem~\cite[Thm.4]{GM},
if the sequence $(b_j(\Lambda M,k))_{j \in \n}$
of Betti numbers of the free loop space
$\Lambda M$ is unbounded then there are infinitely many geometrically distinct
closed geodesics for any Riemannian or Finsler metric.

Moreover if there is a field $k$ such that
the cohomology ring $H^\ast(M;k)$ of $M$ over $k$ has at least two generators, then
the sequence $(b_j(\Lambda M,k))_{j \in \n}$ is unbounded (first this result  was shown for $k=\Q$ the rational field
in~\cite{VS}, and later extended onto all fields of coefficients in \cite{La}).

To derive the case (a) from these results we prove

\begin{proposition}
If $M_1$ and $M_2$ are simply-connected $n$-dimensional manifolds non-homeomorphic to the $n$-sphere, then there is a field $k$ equal to $\Z_p$ for some prime $p$ or to $\Q$ such that the cohomology ring
$H^\ast(M_1 \#M_2;k)$ has at least two generators.
\end{proposition}

This proposition is straightforwardly derived from the following

\begin{lemma}
Let $X$ be an $n$-dimensional simply-connected manifold non-homeomorphic to
the $n$-sphere.
Then at least one of two possibilities holds:

a) $X$ is not a rationally homological sphere, i.e., there is $k<n$ such that
$H^k(X)$ is infinite;

b) the cohomology ring $H^\ast(X;\Z_p)$ has at least two generators
for some prime $p$.
\end{lemma}

{\sc Proof of Lemma.} Let us assume that $X$ is a rationally homological sphere, i.e., a)
does not hold. By the Hurewicz theorem, there is $k$ such that $k>1$,
$H_i(X) = 0$ for $0<i<k$, and $H_k(X) \neq 0$.
Since $X$ is not homeomorphic to the sphere, $k<n$ and, since
$H_\ast(X;\Q) = H_\ast(S^n;\Q)$, $H_k(X)$ is finite. Let us take a prime $p$ such that the order of
$H_k(X)$ is divisible by $p$ and therefore $H_k(X;\Z_p)\neq 0$. Since $\Z_p$ is a field, the groups
$H_i(X;\Z_p)$ and $H^i(X;\Z_P)$ are dual for all $i$ and we conclude that $H^k(X;\Z_p) \neq 0$.
By theorem on universal coefficients, we have
the exact sequence
$$
0 \to \mathrm{Ext}(H_k(X),\Z_p) \to H^{k+1}(X;\Z_p) \to \mathrm{Hom}(H_{k+1}(X),\Z_p) \to 0.
$$
Since $H_k(X)$ has a $p$-torsion, $\mathrm{Ext}(H_k(X),\Z_p) \neq 0$, the exact sequence implies that $H^{k+1}(X;\Z_p) \neq 0$, and we derive that
$H^\ast(X;\Z_p)$ has at least two generators of dimensions $k$ and $k+1$ which implies that b) holds.
Lemma is proved.

Hence, we conclude that on a connected sum of simply-connected manifolds there infinitely many geometrically distinct closed geodesics for all Riemannian and Finsler metrics.

If $M=M_1 \# M_2$, $\pi_1(M_1)$ is finite, and $\pi_1(M_2)=0$, then the universal cover
$\widetilde{M}$ is homeomorphic to the connected sum of the universal cover $\widetilde{M_1}$ of
$M_1$ and $m$ copies of $M_2$ where $m$ is the order of $\pi_1(M_1)$. Since the universal covering
$\widetilde{M} \to M$ is $m$-sheeted the existence of infinitely many geometrically distinct closed geodesics on $\widetilde{M}$ implies the same conclusion for $M$. Therefore the case (a) is justified.

\section{Closed geodesics on $3$-manifolds}
\label{sec:three}
We recall the following statements about manifolds
of dimension $3,$ a general reference is the book~\cite{AFW}
by Aschenbrenner, Friedl and Wilton.
An orientable $3$-manifold is called {\em prime} if it cannot
be decomposed as a non-trivial connected sum.
Hence if $ M= M_1 \# M_2$ then either $M_1$ or $M_2$ is
diffeomorphic to $S^3.$
A $3$-manifold $M$ is called {\em irreducible} if every
embedded $2$-sphere in $M$ bounds a $3$-ball.
Conversely an orientable prime $3$-manifold is
either irreducible or diffeomorphic to $S^1\times S^2.$
Note that we consider only manifolds {\em without boundary,}
whereas in reference \cite{AFW} manifolds in general may have
boundary.
One obtains the following decomposition result,
cf.~\cite[Thm.3.0.1]{AFW}:
\begin{proposition}
\label{pro:decomposition}
Let $M$ be a compact, orientable $3$-manifold.
Then $M$ admits a decomposition
\begin{equation}
\label{eq:decomposition}
M=S_1 \# \cdots \# S_k\#T_1\#\cdots \#T_l\#N_1\#
\cdots\# N_m\,;\,k,l,m \in \n_0
\end{equation}
as a connected sum of orientable prime $3$-manifolds.
Here
\begin{itemize}
\item[(a)] $S_1,\ldots, S_k$ are the prime components of $M$
with finite fundamental group, these are spherical
spaces.
\item[(b)] $T_1,\ldots, T_l$ are
the prime components of $M$ with infinite solvable
fundamental group. Hence $T_i, i=1,\ldots, l$ is either diffeomorphic
to $S^1\times S^2,$ or it has a finite solvable cover
which is a torus bundle.
\item[(c)] $N_1, \ldots, N_m$ are the prime
components of $M$ whose fundamental group
is neither
finite nor solvable.
$N_j, j=1,\ldots,m$
is either hyperbolic, or finitely covered by an
$S^1$-bundle over a surface $\Sigma$ with
negative Euler characteristic
$\chi(\Sigma)<0,$ or has
a non-trivial geometric decomposition.
\end{itemize}
\end{proposition}
It follows that the fundamental group
$\pi_1(M)$ is the free product of the fundamental groups
of the prime components, i.e.
\begin{equation*}
\pi_1(M)=\pi_1(S_1)*\cdots *\pi_1(S_k)*
\pi_1(T_1 )*\cdots *\pi_1(T_l)*
\pi_1(N_1)*\cdots *\pi_1(N_m)\,,
\end{equation*}
cf. Equation~\eqref{eq:fund-group-conn-sum}.
\begin{definition}
Let $M$ be a smooth manifold. The {\em first virtual
Betti number} $vb_1(M)$ is defined as the maximal
first Betti number
$b_1(\widetilde{M})=\rk_{\z}H_1(\widetilde{M};\z)$ of a finite covering $\widetilde{M},$
i.e. $vb_1(M)\in \z^{\ge 0}\cup \{\infty\}$ with
$$vb_1(M):=\sup \{k \in \z; b_1(\widetilde{M})\ge k,
\widetilde{M}\longrightarrow M
\mbox{ is a finite covering}\}.$$
\end{definition}
The following statement was known
in the literature as the
{\em virtual infinite Betti number conjecture.}
It was solved due to solution of
the geometrization conjecture by Perelman,
cf.~\cite[Thm.1.7.6]{AFW}
and the work by
Agol, Kahn-Markovich and Wise
leading to the proof of the {\em virtually
compact special theorem}~\cite[Thm. 4.2.2]{AFW}:
\begin{proposition}{\cite[flowchart p.46, Cor.4.2.3]{AFW}}
\label{pro:infinite-virtual}
If $M$ is a compact, irreducible,
orientable $3$-manifold
such that the fundamental group is neither
finite nor solvable. Then
the virtual first Betti number is infinite, i.e.
$vb_1(M)=\infty.$
\end{proposition}
The statement and the
involved arguments can be read off from
the flowchart~\cite[p.46]{AFW}
together with the
corresponding
{\em Justifications}~\cite[3.2]{AFW}
and~\cite[Cor.4.2.3]{AFW}:
The Geometrization Theorem~\cite[Thm.1.7.6]{AFW} implies that
any compact, orientable, irreducible $3$-manifold
(with empty boundary)
is Seifert fibred, hyperbolic or admits an incompressible
torus. If the manifold is hyperbolic then
$vb_1(M)=\infty.$
This is a consequence of
the above mentioned {\em virtually
compact special theorem}~\cite[Thm. 4.2.2]{AFW}
due to Agol \cite{Agol}, Kahn-Markovich and Wise
and is stated in~\cite[Cor.4.2.3]{AFW}

If the manifold is Seifert fibred it is finitely covered
by an $S^1$-bundle $\widetilde{M}$
over a surface $F$ which has negative
Euler characteristic
$\chi(F)<0$
since by assumption the fundamental
group $\pi_1(M)$ is neither solvable nor finite.
Therefore the covering $\widetilde{M}$ admits an
{\em incompressible torus.}
Hence we are left with the case that
a covering $\widetilde{M}$ of $M$ contains an incompressible torus
$T \subset \widetilde{M}.$
Since $\pi_1(\widetilde{M})$ is neither finite nor solvable
we conclude from~\cite[Thm.1.1]{Lue} that
$vb_1(M)=\infty.$
The statement of Proposition~\ref{pro:infinite-virtual}
can also be found in~\cite[Thm.1.1]{LN}.

It is easy to see that a compact Riemannian
manifold with
virtual Betti number at least $2$ carries infinitely
many closed geodesics:
\begin{proposition}
\label{pro:virtual-number}
Let $M$ be a compact Riemannian manifold with
a Riemannian or Finsler metric.
If the virtual Betti number satisfies $vb_1(M)\ge 2$
then for any $k \le vb_1(M)$
the number $N(t)$ of geometrically distinct closed geodesics
of length $\ge t$ satisfies:
$\liminf_{t\to \infty}N(t)/t^k>0\,.$
In particular there are infinitely many geometrically
distinct closed geodesics.
\end{proposition}
\begin{proof}
Let $F: \widetilde{M}\longrightarrow M$ be a finite
covering of order $r \ge 2$ with $b_1(\widetilde{M})\ge k$
and let $\widetilde{g}$ be the induced Riemannian
metric on $\widetilde{M}$ for which the
finite group of deck transformations $\Gamma$ acts by
isometries. For a closed geodesic $\widetilde{c}:
\sone \longrightarrow \widetilde{M}$ on $\widetilde{M}$
the set $\Gamma (\widetilde{c}(\sone))
=\{\gamma (\widetilde{c}(t)); \gamma \in \Gamma, t\in
\sone\}$ consists of at most $r=\# \Gamma$ distinct closed
geodesics which are mapped under $F$ onto a single
closed geodesic of length
$\le L(\widetilde{c}).$ Since $b_1(\widetilde{M})\ge k$
we obtain for the number
$\widetilde{N}(t)$ of closed geodesics on $\widetilde{M}$
of length $\le t:$
$\widetilde{N}(t)\ge \widetilde{\lambda}_k \, t^k$
for some $\widetilde{\lambda}_k>0,$
cf. \cite[\S 5]{Ba} or \cite[p.55]{Ball}.
Hence the number $N(t)$ of closed geodesics of length
$\le t$ on $M$ satisfies:
$N_g(t)\ge (\widetilde{\lambda}_k/r)\, t^k.$
\end{proof}
\begin{theorem}
\label{thm:three-dim-cases}
Let $M$ be a compact, orientable $3$-manifold
with infinite fundamental group $\pi_1(M)$ and
with prime decomposition given in
Proposition~\ref{pro:decomposition}.
Let $M$ carry a Riemannian or Finsler metric. Then we
obtain the following estimates for the
number $N(t)$ of geometrically distinct closed geodesics
of length $\le t:$
\begin{itemize}
\item[(a)] If the manifold is not prime
and is not diffeomorphic to $\R P^3\#\R P^3$
then
 $\liminf_{t\to \infty} \log(N(t))/t>0\,.$
\item[(b)] If the manifold is prime and if the
fundamental group is neither finite nor solvable then
for any $r\ge 1:$
 $\liminf_{t\to \infty} N(t)/t^r>0\,.$
\item[(c)] If
$M$ is prime and has an infinite solvable
fundamental group or if
$M= \R P^3 \# \R P^3$
then
$\liminf_{t\to \infty} N(t)\log(t)/t>0\,.$
\end{itemize}
\end{theorem}
\begin{proof}
(a)
If the manifold is not prime and not diffeomorphic to
$\R P^3 \# \R P^3$ then
the fundamental group $\pi_1(M)$
is isomorphic to the free product $G_1\ast G_2$ of
two non-trivial groups $G_1, G_2$ where at least one of the groups
has at least three elements. Then the claim follows from
Theorem~\ref{thm:conn-intro}.

\smallskip

(b) If the manifold is prime and if
$\pi_1(M)$ is neither finite nor solvable we conclude from
Proposition~\ref{pro:infinite-virtual}:
 $vb_1(M)=\infty\,.$
Then the claim follows from
Proposition~\ref{pro:virtual-number}.

\smallskip

(c)
If the manifold is prime and the
fundamental group $\pi_1(M)=\pi_1(T_1)$
is infinite and
solvable the conclusion is due to
Taimanov~\cite[Thm.3]{T1993}.

Let $M=\R P^3 \# \R P^3$ i.e. $M$ is the connected sum
of two real projective $3$-spaces. Then
$\widetilde{M}=S^1 \times S^2$ is a twofold cover,
its fundamental group
$\pi_1(\widetilde{M})\cong \z $
is a normal subgroup of index $2$ of the free product
$\z_2 \ast \z_2,$ cf. the Proof of
Theorem~\ref{thm:conn-intro}.
Since $\pi_1(\widetilde{M})\cong \z $
the claim follows from
a result by Bangert-Hingston~\cite[Thm.]{BH},
cf. Theorem~\ref{thm:conn-intro} (a).
\end{proof}
\begin{remark}
\label{rem:two}
\rm
If the $3$-manifold $M$ carries a hyperbolic metric then
for any Riemannian metric $g$ on $M$ the number $N(t)$ of
geometrically distinct closed geodesics with length $\le t$
satisfies
$ \liminf_{t\to \infty}\log(N(t))/t>0\,,$
cf. for example~\cite[Thm.A]{Ka}.
\end{remark}
As a consequence of Theorem~\ref{thm:three-dim-cases}
we obtain
the {\bf Proof of Theorem~\ref{thm:three-dimension}} stated
in the Introduction:
\begin{proof} We can assume
without loss of generality that
$M$ is orientable. If $M$ is non-orientable we
use the twofold orientation covering $\widetilde{M}.$
If $\pi_1(M)$ is infinite also $\pi_1(\widetilde{M})$
is infinite.
Then the claim follows from Theorem~\ref{thm:three-dim-cases}.
\end{proof}
To include also the case of a finite and possibly trivial
fundamental group in dimension three we have to restrict
to a generic metric:
We call a Riemannian metric $g$ {\em twist metric} if either
all closed geodesics are hyperbolic or if there
exists one closed geodesic
of twist type satisfying the
assumptions of the Birkhoff-Lewis fixed point theorem.
For Riemannian metrics this
is a $C^k$-generic condition for $k\ge 4,$ cf.
\cite{KT} or \cite[(3.11)]{Ba}.

Then we obtain the
{\bf Proof of  Theorem~\ref{thm:three-dimensionA}}
stated in the Introduction:
\begin{proof}
Because of Theorem~\ref{thm:three-dimension}
we only have to consider the case of a finite
trivial fundamental group including the case
of a trivial fundamental group. Hence the statement
follows from the corresponding statement for
the universal covering space which is the
$3$-sphere by the
{\em Elliptization Theorem}~\cite[Th.1.7.3]{AFW}.
If there exists a closed geodesic of twist type
then the proof of the Birkhoff-Lewis fixed
point theorem shows that the function $N(t)$ satisfies
$\liminf_{t\to \infty}N(t) \log(t)/t >0.$
The same estimate also holds for a Riemannian metric
on $S^3$ all of whose closed geodesics are hyperbolic.
This was shown by Hingston~\cite[Thm.6.2]{Hi}.
\end{proof}
\begin{remark}
\label{rem:surface}
\rm
For a closed surface, i.e. a compact
$2$-manifold
the number $N(t)$ of closed geodesics
of an arbitrary Riemannian metric satisfies
$\liminf_{t\to \infty}N(t) \log(t)/t>0.$
If the fundamental group of a closed surface
is infinite then the function
$N(t)$ grows at least quadratically since the
its first virtual Betti number is at least $2,$
cf. Proposition~\ref{pro:virtual-number}.
This also holds for any Finsler metric.
For a Riemannian metric on $S^2$ the function $N(t)$
grows like the prime numbers, cf. \cite[Thm.]{Hi1993}.
But note that there are non-reversible Finsler metrics on
$S^2$ with only two closed geodesics, the
geometry of these examples
first introduced by Katok is discussed by
Ziller in~\cite{Zi}.
\end{remark}

\medskip

\begin{thebibliography}{999999}
\bibitem{Agol}
I. Agol,
The virtual Haken conjecture.
 With an appendix by Agol, Daniel Groves, and Jason Manning.
Doc. Math. 18 (2013), 1045--1087.
\bibitem{AFW} M.Aschenbrenner, S.Friedl \& H.Wilton:
$3$-Manifold Groups, EMS series lect. math. vol. 20,
Europ. Math. Soc. 2015
\bibitem{Ball} W.Ballmann,
Geschlossene Geod\"atische auf Mannigfaltigkeiten
mit unendlicher Fundamentalgruppe.
Topology 25 (1986) 55-69.
\bibitem{BTZ1981} W.Ballmann, G.Thorbergsson \&
W.Ziller, Closed geodesics and the fundamental
group. Duke Math.~J. 48 (1981) 585-588
\bibitem{Ba} V.Bangert,
Geod{\"a}tische Linien auf Riemannschen Mannigfaltigkeiten.
Jber.d.Dt.Math.Verein. 87 (1985) 39-66
\bibitem{Ba2}
V.Bangert, On the existence of closed geodesics on two-spheres, Internat.
J. Math. 4 (1993)1–10
\bibitem{Br} G.E.Bredon: Topology and Geomety, Graduate Texts Math. 139,
2nd printing,
Springer, Berlin Heidelberg New York 1995
\bibitem{BH} V.Bangert \& N.Hingston,
Closed geodesics on manifolds with infinite abelian
fundamental group, J.~Differential Geom. 19 (1984) 277-282
\bibitem{GM}
D.Gromoll \& W.Meyer,
Periodic geodesics on compact Riemannian
manifolds,
J.Differential Geom. 3 (1969) 493-510
\bibitem{Gr} M.Gromov,
Three remarks on geodesic dynamics and fundamental group,
preprint State Univ. New York, Stonv Brook (1976),
reprinted: L'Enseignement Math. 46 (2000) 391-402.
\bibitem{Hi} N.Hingston,
Equivariant Morse theory and closed
geodesics,
 J.~Differential~Geom. {\bf 19} (1984) 85--116
\bibitem{Hi1993}
N.Hingston:
On the growth of the number of closed geodesics on the
two-sphere, Intern.Math.Res.Notices 9 (1993) 253-262.
\bibitem{Ka} A.Katok,
Entropy and closed geodesics,
Ergod.Th. \& Dynam.Sys. 2 (1982) 339-367
\bibitem{KT} W.Klingenberg \& F.Takens,
Generic properties of geodesic flows,
Math. Ann. 197 (1972) 323-334
\bibitem{La} P.Lambrechts,
The Betti numbers of the free loop space of
a connected sum,
J.~London~Math.Soc. 64 (2001) 205-228

\bibitem{LN} T.J.Li \& Y.Ni,
Virtual Betti numbers and virtual symplecticity of
$4$-dimensional mapping tori,
Math.~Z. 277 (2014) 195-208
\bibitem{Lue} J.Luecke, Finite covers of $3$-manifolds containing
essential tori, Trans.~Amer.~Math.~Soc. 310 (1988) 381-391
\bibitem{LS} R.C.~Lyndon, P.E.~Schupp,
Combinatorial group theory, Ergebn. Math. 89,
Springer, Berlin, Heidelberg, New York 1977
\bibitem{Ma} R. Ma\~n\'e,
On the topological entropy of geodesic flows,
J.Differential Geom. 45 (1997) 74-93
\bibitem{PP}
G.Paternain \& J.Petean,
On the growth rate of contractible closed geodesics
on reducible manifolds,
Geometry and Dynamics,
191-196, Contemp.Math. 389,
Amer.Math.Soc., Providence, RI, 2005
\bibitem{Ra1}
H.B.Rademacher,
On the average indices of closed geodesics,
J.~Differential~Geom. 29 (1989) 65-83
\bibitem{T1985} I.A.Taimanov, Closed geodesics on non-simply-connected
manifold, Uspekhi Mat. Nauk 40:6 (1985) 157--158
= Russian Math. Surveys 40:6 (1985) 143--144
\bibitem{T1993}
I.A.Taimanov, Closed geodesics on non-simply-connected
manifolds, Sibirskii Mat. Zhurnal 34:6 (1993) 170--178
=
Siberian Math. J. 34:6 (1993) 1154–1160
\bibitem{T2010}
I.A.Taimanov,
The type numbers of closed geodesics.
Regul.Chaot.Dyn. 15 (2010) 84-100
\bibitem{T} M.Tanaka,
Closed geodesics on compact Riemannian manifolds with
infinite fundamental groups.
Proc.Fac.Sci.Tokai Univ. 20 (1985) 1-12
\bibitem{VS}
M. Vigue-Poirrier, D. Sullivan, The homology theory of the closed geodesic problem, Journal of Diff. Geometry 11 (1976), 633-–644
\bibitem{Zi} W.Ziller, Geometry of the Katok examples,
Ergod.Th.\& Dyn.Syst. 3 (1982) 135--157
\end{thebibliography}
\end{document}